\documentclass{article}
\usepackage{amsmath, amssymb, amsthm}
\usepackage{tikz}
\usepackage[utf8]{inputenc}
\usepackage{tabularx}
\usepackage[edges]{forest}
\usepackage[most]{tcolorbox}
\usepackage{tabularx}
\usetikzlibrary{graphs, graphs.standard}

\newtheorem{definition}{Definition}
\newtheorem{theorem}{Theorem}

\newtheorem{lemma}{Lemma}
\title{Model Theory of Ultrafinitism II: \\
Deconstructing the Term Model\\
(First Draft)}

\author{Mirco A. Mannucci, \\HoloMathics LLC, \\mirco@holomathics.com}
\date{\today}

\begin{document}

\maketitle

\begin{abstract}
This paper presents a novel possible worlds semantics, designed to elucidate the underpinnings of ultrafinitism. By constructing a careful modification of the well-known  Kripke models for inuitionistic logic, we seek to extend our comprehension of the ultra-finite mindset. As it turns out, the passage from  standard constructivist mathematics to the ultrafinite  is in a sense an operation of deconstruction of familiar mathematical entities, most notably clear when it comes to $N$.
\end{abstract}

%\keywords{Ultrafinitism  \and Forcing \and %Model Theory}

\section{Introduction}

\begin{quote}
Q: You want the constructivists worry about feasibility?\\
A: Yes, feasibility is the real issue. It seems to me that the constructivist movement lost the pioneering spirit. In computer science, constructivists did not play the role one could expect them to play. They remained largely satisfied with constructivity on the level of recursive functions and did not refine the notion of construction sufficiently. Of course that task is very hard and the constructive community was never very large. Still, they could have been the ones to pioneer computational complexity theory. They could have been the first to criticize the existing complexity theory for ignoring too often the feasibility issue.\\
Q: But the feasibility concept is hard. It is hard to define in theoretical terms what is feasible.\\
A: True, but this is also a challenge.

\emph{Maybe a new constructivism will arise to meet the challenge}.
\end{quote}
-- Yuri Gurevich, "Platonism, Constructivism, and Computer Proofs vs. Proofs by Hand" \cite{Gurevich}
\\
\\
In mathematical logic, Kripke models have provided a robust structure for understanding intuitionistic logic. They give us a means of grasping  the notion of truth as it is understood from an intuitionistic perspective, not as a  static immutable concept, but rather evolving along ever growing mathematical experience. By generalizing the notion of model to a 'Kripke frame' of possible worlds ordered by inclusion, Kripke structures allow us to explore how our knowledge evolves and how this impacts the validity of propositions.

.

Given the unquestionable utility of the Kripkean analysis in shedding light on the intuitionistic perspective, it is natural to ask whether somewhat similar structures, but  equipped with more stringent constraints,  could help us come to grips with the most radical version of constructivism, namely ultrafinitism (a form of radical  finitism that rejects or cast doubts on   infinite totalities, not just those  \emph{in actu}, but also  \emph{in potentia}). 

Despite its growing philosophical appeal in some quarters, \footnote{A note en passant: whereas only fifteen years ago ultrafinitism was a little more than a label for a remote land in the foundation of mathematics galaxy, in recent years some of its themes crop up in many unexpected places, for instance in the foundational musings on potentialism of Joel David Hamkins (see \cite{joel}, in a note by Field's Medal Tim Gowers on Mathoverflow (see \cite{gowers}), and  a few other spots. Why? No doubt the immense achievements of automated theorem proving, deep generative AI, and a lot of other successes in computational mathematics and computer science create a urge to understand "the mind of the  machines", which no Turing analysis can fully capture (machines do not have unlimited time, memory, computational strength, etc). But there seem to be  internal driving forces to Foundation of Mathematics itself, though not as easy to pin down (a partial clue, though, is that in  a era which is prone to multiverses, there is an internal push to make even the arithmetical structure just one of several non-isomorphic options).  Be that as it may, it is clear that the topic needs further elucidations.} ultrafinitism remains poorly understood and controversial. Up to now, it has sorely lacked both a fully convincing proof theory and a model theory.

In this paper, we present an attempt to construct such a model theory  for first-order logic -  \textbf{Esenin-Volpin models}. Esenin-Volpin models modify the Kripke framework to encompass a concept of feasibility, with the aim of providing a formal structure that can capture the key intuitions of radical constructivism. This is part of a larger project of investigating various model constructions for ultrafinitistic logic, as announced in \cite{MannucciCherubin} .

In the next installment of the series, we will  leverage the other popular semantics for constructive mathematics, namely realisability, to carve out another model theory for ultrafinitism. 

We shall make a few assumption in the following:
\begin{itemize}

\item the logic (ie the set of logical rules and axiom schemas) shall be IPC, Intuitionistic Predicate Calculus. 
\item 
We shall also assume that the underlying deduction system is the normalized LJ, with cut-free proofs (this second restriction is strictly not necessary, as there is always a super-exponential transformation which normalizes deductions via Gentzen, but will make our presentation considerably easier.)
\item 
Though the proposed framework could be adapted to handle  recursively axiomatisable theories, in the following we will only deal with finitely axiomatised ones. 

\end{itemize}

Unlike other research in bounded resources computation, for instance linear logic, here what changes is not the axioms or deduction rules, rather the very notion of deduction. It is the complexity of deduction trees that matter. As we have already pointed out in \cite{MannucciCherubin}, this step entails many consequences, for instance this one: whereas in both classical and constructivistic universes contradiction amount to  a death sentence, here it does not. Only feasibly achieved contradiction is. \footnote{ We hasten to point our that our remedy is not paraconsistent logic, ie changing somehow the logical rules (like the infamous \textit{ex falso quodlibet}). Rather, we simply consider the cost of producing a contradiction, and at what level it may impact our constructions. However, we envision a slightly different version of Volpin models which includes paraconsistent nodes, ie nodes forcing falsum. These nodes would lie above the consistency depth of the model. } 

\section{Esenin-Volpin Models: Motivation and Formal Definition}

We assume the reader to be familiar with standard Kripke semantics for intuitionistic first order logic, as it is to be found for instance  in \cite{Moshkovakis}. and \cite{Kuznetzov}. As a matter of fact, in the ensuing discussion, we will use the material in Kuznetsov as a template for our own definitions and results.

\subsection{Motivation}
Classical Kripke models do not readily lend themselves to the ultrafinite paradigm. 

This is mainly due to three intrinsic characteristics of the Kripkean framework:

\begin{enumerate}

\item \textbf{They lack a concept of resource-awareness}. A fundamental aspect of ultrafinitism is the emphasis on \emph{feasibility in computation}. The existence and verifiability of mathematical objects is subject to restrictions on the resources required to construct  and manipulate these objects. This consideration is not captured by the traditional Kripke models.
\item \textbf{The possible world structures in a Kripke frame can become infinite} (it is well known that the Finite Model Theorem fails in the kripke semantics for first order logic. Also, it is perhaps worth pointing out that in kripke models of HA, Heyting Arithmetics, already the ground zero world contains a copy of N). 
However, ultrafinitism takes a restrictive stance towards not just infinite totalities, but also extremely large finite ones. 
\item \textbf{They provide a global interpretation of a theory}. The third point where the ultra-finite perspective parts way not only with intuitionism, but pretty much the full gamut of traditional constructivism, is this: \emph{ a shift from global models to partial models}. Only a finite and complexity bounded set of formulas and terms is evaluated. That may sound strange, but a moment of thought will convince the Reader that assuming that for instance all primitive recursive functions are given in a single shot is tantamount to assuming that $N$ exists: terms are created by a limited resource entity. 
\end{enumerate}

These considerations motivate the development of Esenin-Volpin models,  Kripke-like structures  that incorporate feasibility considerations and restricts the possible world structures to finite ones. Furthermore, the forcing relation will be applicable only to a finite well-defined set of sentences.

As we shall see in a moment, the very rules of forcing will be altered to fit the new perspective.

\subsection{Esenin-Volpin Frames and the Forcing Relation}
First things first: let us introduce Esenin-Volpin Models and their notion of forcing.
But before we give the formal definition, let us provide the Reader with the proper intuition:

\begin{tcolorbox}
[colback=black!5!white,colframe=black!75!black,title= Heuristic of Esenin-Volpin Models]
Van Dalen points out in his excellent introduction to intuitionistic logic  \cite{Vandalen} that nodes are to be thought of as knowledge states of some ideal mathematician. Here they are the same, with one glaring difference: the mathematician is some finite resources individual/machine, and he/she/it can handle only finite bits of information, at a price. 
\\\\
Each steps in his/her/it deductions costs something, even going from the established knowledge of $A$ and $B$ to their conjunction. Moreover, unlike Brouwer's ideal math fellow, who  never dies and continues doing math till the end of times, our  has a limited lifespan: the models have a maximum finite depth.
\\\\
What happens beyond the maximal depth simply does not exists: perhaps a  contradiction is lurking behind the horizon, but it is not recorded. If one keeps in mind this picture, the new forcing conditions will look -we hope- quite natural.

\end{tcolorbox}

\begin{definition}[\textbf{Volpin Frames}]
Let \( R_1, \ldots, R_s \) be distinct predicate letters of \( L(Pd) \), \( f_1, \ldots, f_t \) be distinct function symbols of \( L(Pd) \), and \( c_1, \ldots, c_k \) be distinct constant symbols of \( L(Pd) \). Here, \( R_i \) is \( n_i \)-ary (\( 1 \leq i \leq s \)) and \( f_j \) is \( m_j \)-ary (\( 1 \leq j \leq t \)).

A Volpin frame \( \mathcal{V} \) over \( R_1, \ldots, R_s \) and \( f_1, \ldots, f_t \) is an \( (s + t + k + 3) \)-tuple
\[
\mathcal{V} = ((\mathcal{W}, \leq), D, \delta, \chi_1, \ldots, \chi_s, \lambda_1, \ldots, \lambda_t, \kappa_1, \ldots, \kappa_k)
\]
where 
\begin{itemize}
    \item \( (\mathcal{W}, \leq) \) is a finite rooted tree.
    \item \( D \) is a family of finite sets indexed by nodes of \( \mathcal{W} \).
    \item \( \delta \) assigns to each node \( w \) of \( \mathcal{W} \) a domain \( \delta(w) \subseteq D_w \).
    \item For each \( 1 \leq i \leq s \), \( \chi_i \) is an \( (n_i + 1) \)-ary function from \( \mathcal{W} \times D^{n_i} \) to \( \{0, 1\} \).
    \item For each \( 1 \leq j \leq t \), \( \lambda_j \) is an \( (m_j + 1) \)-ary function representing \( f_j \) such that if \( t \) is a term and is present at world \( w \), then \( \lambda_j(t) \) exists in every world \( w_1 \) with \( w_1 \geq w \).
    \item For each \( 1 \leq l \leq k \), \( \kappa_l(w) \) represents the interpretation of the constant symbol \( c_l \) in world \( w \).
\end{itemize}

We shall denote \(R(w) =\{ v | w\leq w  \} \), in other words, the set of all extensions of $w$
\end{definition}

A comment is in order: notice that worlds are not necessarily closed with respect to arbitrary terms. Each world must be closed only  with respect to terms which can be generated via the functions in the language in one steps from terms already existing in previous worlds domains. This is in keeping with the overarching principle that everything costs some effort, even term generation. 

\begin{definition}[\textbf{Esenin-Volpin Forcing}]
The forcing relation \( \models \) between worlds and formulas is defined as follows:
\begin{itemize}
    \item \( w \models \phi \land \psi \) if and only if there exist \( w_1, w_2 \) such that \( w_1, w_2 \prec w \) and \( w_1 \models \phi \) and \( w_2 \models \psi \).
    \item \( w \models \phi \lor \psi \) if and only if there exists a \( w_1 \) such that \( w_1 \prec w \) and either \( w_1 \models \phi \) or \( w_1 \models \psi \).
    \item For each \( w \), \( w \not\models \bot \).
    \item \( w \models \phi \to \psi \) if and only if for each \( w_1 \) such that \( w \preceq w_1 \), if there exists a \( w_2 \prec w_1 \) such that \( w_2 \models \phi \), then \( w_1 \models \psi \).
    \item \( w \models \exists x\, \psi(x) \) if and only if there exists a \( w_1 \) such that \( w_1 \prec w \) and an \( a \in D_{w1} \) such that \( w_1 \models \psi[a] \).
    \item \( w \models \forall x\, \psi(x) \) if and only if for every \( w_1 \) such that \( w \preceq w_1 \) and for every \( a \in D_{w1} \), \( w_1 \models \psi[a] \).
\end{itemize}

\end{definition}

Now we come to a crucial point, namely validity in a model. In Kripke models, to be true in a model  means to be forced at every node. But clearly, this demand is too restrictive in this new context: forcing rules "push up" validity, as sentences become more and more complicated. So, we need to rethink this notion a bit:

\begin{definition}[\textbf{Validity}]

Given a Esenin-Volpin model \( M \), a $k$ less or equal \(M\) maximal depth,  and a sentence \( \phi \), we say that \( M \) \textit{k-forces} \( \phi \) (denoted \( M \underset{k}{\Vdash} \phi \)) if for every world \( w \) in \( M \) that is at a depth of \( k \), \( w \) forces \( \phi \).Formally, 
\[ M \underset{k}{\Vdash} \phi \iff \forall w \in M \text{ with depth of } w = k, w \Vdash \phi \]
\end{definition}
Although Volpin Forcing is quite different from the  Kripke-Cohen one, it still enjoys monotonicity, as it should:

\begin{lemma}[\textbf{Monotonicity}]
In a Esenin-Volpin model, if a world $w$ forces a formula $\phi$, then any successor world $w'$ of $w$ (i.e., a world such that $w < w'$) also forces $\phi$.
\end{lemma}

\begin{proof}

Let's prove the proposition by induction on the structure of \( \phi \).

\begin{enumerate}
    \item For atomic formulas: This is given by the properties of \( \chi_i \) for \( R_i \) and the definition of function symbols, constants and terms.
    
    \item \(\phi = \phi_1 \land \phi_2\): If \( w \models \phi_1 \land \phi_2 \), then there exist worlds \( w_1, w_2 \) such that \( w_1, w_2 \prec w \), \( w_1 \models \phi_1 \), and \( w_2 \models \phi_2 \). But then $w_1, w_2 < w'$. Hence, \( w' \models \phi_1 \land \phi_2 \).
    
    \item \(\phi = \phi_1 \lor \phi_2\): If \( w \models \phi_1 \lor \phi_2 \), then there exists a world \( w_1 \) such that \( w_1 \prec w \) and either \( w_1 \models \phi_1 \) or \( w_1 \models \phi_2 \). Again, $w_1 < w'$. Hence, \( w' \models \phi_1 \lor \phi_2 \).

    \item \(\phi = \psi_1 \to \psi_2\): Assume \( w \models \psi_1 \to \psi_2 \) and there exists a \( w_2 \prec w' \) such that \( w_2 \models \psi_1 \). Then by the definition of the implication in our Esenin-Volpin model and the inductive hypothesis, \( w' \models \psi_2 \). Thus, \( w' \models \psi_1 \to \psi_2 \).
    
    \item \(\phi = \exists x\, \psi(x)\): If \( w \models \exists x\, \psi(x) \), then there exists a world \( w_1 \) and an element \( a \) such that \( w_1 \prec w \) and \( w_1 \models \psi[a] \). By the inductive hypothesis, \( w' \models \psi[a] \). Hence, \( w' \models \exists x\, \psi(x) \).
    
    \item \(\phi = \forall x\, \psi(x)\): If \( w \models \forall x\, \psi(x) \), then for every world \( w_1 \) such that \( w \preceq w_1 \) and for every \( a \in D_{w1} \), \( w_1 \models \psi[a] \). By the inductive hypothesis, for every \( a \in D_{w'1} \), \( w' \models \psi[a] \). Hence, \( w' \models \forall x\, \psi(x) \).
\end{enumerate}

\end{proof}

\section{Completeness theorem for k-consistent theories}

In this section, we introduce the concept of k-consistency and present the soundness and completeness theorems of Esenin-Volpin models of suitable depth for k-consistent theories.

\subsection{k-consistent theories}

\begin{definition}[\textbf{k-consistency}]

A theory $T$ is said to be k-consistent if and only if there is no contradiction in $T$ that can be proved with a cut-free  proof tree of depth less or equal $k$.
\end{definition}

Admittedly, the notion of complexity highlighted above is quite coarse-grained, and would be inadequate in modeling actual computations of  real machines. It is simply an approximation which, it is hoped, will provide enough clarity on the direction we are taking.

\subsection{ Soundness and Completeness}

\begin{lemma}[\textbf{Soundness} ]
Let  $\phi$ be a formula  which has a cut-free  proof of complexity less than $k$. Then for every  Esenin-Volpin model $M$ of depth at least $k$, we have $M \Vdash_k \phi$.
\end{lemma}
\begin{proof}
We will use induction on the complexity of the proof of $\phi$.

\textbf{Base case}: 
\begin{itemize}
    \item \textbf{Atomic Formulas}: Just falling out the definition for forcing of basic formulas.
\end{itemize}

\textbf{Inductive step}:

\begin{itemize}
    \item \textbf{Conjunction}: Suppose \( \phi \) is of the form \( \psi \land \chi \). If there's a proof of \( \phi \) of length \( k \), there are shorter proofs of its constituents. But then by inductive hypothesis,  there exist worlds \( w_1 \) and \( w_2 \) with depth less than \( k \) that force \( \psi \) and \( \chi \) respectively. Thus, world \( w \) at depth \( k \) forces \( \psi \land \chi \).

    \item \textbf{Disjunction}: If \( \phi \) is \( \psi \lor \chi \), and it has a proof of length \( k \), then either \( \psi \) or \( \chi \) has a proof  one step shorter. Thus, by induction, one of them is forced at a world with depth strictly less than \( k \), which means \( \phi \) is forced at depth \( k \).

        \item \textbf{Implication}: Assume  that \( w \) is not a leaf node (else by the forcing condition of implication is trivially true).  Consider any world \( w_1 \) greater than \( w \), say of depth  of depth \( k_1 \), if \( \psi \) is forced by any previous world, then  by inductive hypothesis there is a proof of \( \psi \)shorter than \(k_1 \). Now, join the implication and this proof to obtain a proof shorter or equal to the depth of  \( w_1 \) of \( \chi \).  But then \(w\)
  forces the implication by the forcing definition.

    \item \textbf{Existential Quantifier}: If \( \phi \) is \( \exists x \, \psi(x) \), and it has a proof of length \( k \), then there's a prior proof of \(\psi)(t)\) for some term \(t\) (it could be a constant).  Now, per inductive hypothesis, there is a  prior world \( w_1 \) and an element \( a \) in its domain such that \( w_1 \) forces \( \psi(a) \). So, by the forcing rule for the existential, \(w\) forces  \( \phi \)

    \item \textbf{Universal Quantifier}: For \( \phi \) as \( \forall x \, \psi(x) \), for every subsequent world \( w_1 \) and for all elements in its domain, there is a proof of  \( \psi(a) \) of seize max $k+1$, so by induction \(w_1\) forces  \( \psi(a) \). Which implies, via the forcing rules, our result.
\end{itemize}

\end{proof}

In order to prove completeness for Esenin-Volpin models, we follow an approach akin  to the completeness theorem of Kripke's semantics for first order intuitionistic  logic (in fact our proof will be in a sense the "miniaturization" of the standard one): we shall  construct a canonical model $M_0$, which will be our universal counterexample provider. $M_0$ will be a syntactic model, assembled out of theories which will be consistent up to some depth.

Let us start: assume \(0 \leq m \leq k\). 
Given a large finite set \( S_0 \) of new constants,  and the  language's signature  \( L_{\Omega} \), where we indicate as \(S_{\Omega}\) its constants, 
We extend \( L_{\Omega} \) with constants from \( S \) (a finite subset of \( S_0 \)) and denote this extended language of closed formulas of structural  complexity at most $m$ as \( \text{m-CF}(L_{\Omega + S}) \). 

\begin{definition}

\textbf{Bi-theories of depth m} are  triples of the form \( (S, \Gamma, \Delta) \), where \( \Gamma, \Delta \subseteq \text{m-CF}(L_{\Omega + S}) \), and $S_{\Omega} \subseteq S \subset S_0$.  A bi-theory is\textbf{ replete} iff:

\begin{description}
    \item[m-Consistent:] A bi-theory \( (S, \Gamma, \Delta) \) is called \( m \)-consistent if there do not exist finite sets \( \Gamma_0 \subset \Gamma \) and \( \Delta_0 \subset \Delta \) such that there is a cut-free proof of \(\land \Gamma_0 \rightarrow \lor \Delta_0 \) in \( m \) steps or fewer.
    
    \item[m-Complete:] \( (S, \Gamma, \Delta) \) is \( m \)-complete if, for every closed formula \( \phi \in \text{m-CFmL} + S \), either  \( \phi \in \Gamma \) or \( \phi \in \Delta \) .
    
    \item[m-Existential Closed:]
    if  \( \phi = \exists x \psi(x) \in \Gamma)\), there is already a term $t$ of complexity less than $m$ such that \(\phi(t) \in \Gamma\) and $t \in S$.
    
\end{description}

The\textbf{ canonical model} \( M_0 \) of depth $k$is given by the structure \( \langle W_0, R_0, D_0, \alpha_0 \rangle \), where:

\begin{itemize}
    \item \( W_0 \) is the set of all small \( m \)-replete bi-theories for all $ 0\leq m \leq k$; (notice that this gives us automatically a depth for each of them) 
    \item For bi-theories \( (S_1, \Gamma_1, \Delta_1) \) and \( (S_2, \Gamma_2, \Delta_2) \), the relation \( R_0 \) holds between them iff \( S_1 \subseteq S_2 \), $S_2$ contains all terms generated in one steps from elements of $S_1$, and \( \Gamma_1 \subseteq \Gamma_2 \);
    \item For each world \( w = (S,\Gamma,\Delta) \), the domain is given by \( D(w) = S \);
    \item For each predicate symbol \( P \) and elements \( a_1, \ldots, a_{v(P)} \in S \), the interpretation function \( \alpha_0(w)(P)(a_1, \ldots, a_{v(P)}) = 1 \) if and only if \( P(a_1, \ldots, a_{v(P)}) \in \Gamma \).
\end{itemize}
\end{definition}

In this structure, the worlds are replete  bi-theories, the accessibility relation \( R_0 \) reflects the extension of constants and formulas as well as the extension of computational power, \( D_0 \) defines the domain for each world, and \( \alpha_0 \) provides an interpretation of predicates at each world.The roles of \(\Gamma \) and 
\(\Delta\) are respectively to be the theory (ie the forced formulae) and the non-theory (ie the formulae which are not forced) at the given world. 

Just like in the proof of completeness for Kripke models, we need two lemmas: The first lemma tells us that we can broaden an initial seed  bi-theory to a maximal $m$-consistent one, where $m$ is the depth of the world. 

\begin{lemma}[\textbf{Saturation Lemma}]
Let \( (S, \Gamma, \Delta) \) be a small \( m \)-consistent bi-theory. Then, there exists a small \( m \)-consistent, \( m \)-complete, \( m \)-\(\exists\)-complete bi-theory \( (S_0, \Gamma_0, \Delta_0) \) such that \( S \subseteq S_0 \), \( \Gamma \subseteq \Gamma_0 \), and \( \Delta \subseteq \Delta_0 \).
\end{lemma}
\begin{proof}

 Start with \( (S, \Gamma, \Delta) \). Now, enumerate all formulas in \( \text{m-CF}(L_{\Omega + S}) \) and go through the list.
If  \(\phi_i\), is already in either \( \Gamma \) or \(\Delta\), leave it there. Else, try to add it to \( \Gamma\). If the new augmented \( \Gamma \) is $m$- consistent, it replaces the old one. Else, throw \( \phi_1\) in the \(\Delta \) bucket. 

As for existential closure, just ensure that if in \( \Gamma\) there is a formula \(\exists x \psi(x) \), a suitable witness $t$ is already in $S$. If it isn't, grab a fresh constant from the pool of constants and augment your $S$ as needed and add to \(\Gamma \) \(\psi(t) \).

\end{proof}
 Before going into the next Lemma, we point out this fact, which descends immediately from the previous lemma: If \( (S, \Gamma, \Delta) \)  is $m$-replete, then if \( \Gamma \Vdash_m \phi\), \(\phi \in \Gamma\): in other words,\textit{ worlds of the canonical model are $m$-deductively closed}. 
 
\begin{lemma}[\textbf{Main Semantic Lemma}]
In the canonical model \( M_0 \), for any closed formula \( \phi \), \( (S, \Gamma, \Delta) \models \phi \) if and only if \( \phi \in \Gamma \).
\end{lemma}

\begin{proof}
We prove this lemma by induction on the structure of \( \phi \). 

\textbf{Base case:} If \( \phi \) is atomic, the result is directly given by the definition of the canonical model.

\textbf{Inductive step:} 

1. \( \phi = \psi_1 \land \psi_2 \):

 Let us assume that \( (S, \Gamma, \Delta) \models \psi_1 \land \psi_2 \). Then by forcing rules, there are two lower worlds \(\Gamma_1\) and  \(\Gamma_2\) such that 
\( (S_1, \Gamma_1, \Delta_1 ) \models \psi_1 \) and \( (S_2, \Gamma_2, \Delta_2 ) \models \psi_2 \). But then by inductive hypothesis \( \psi_1 \in \Gamma_1\) and \( psi_2 \in \Gamma_2\), and thus \( \psi_1, \psi_2 \in \Gamma\). By deductive closure (we in fact need one step deduction), we can infer that \( \psi_1 \land \psi_2 \in \Gamma \). Now, let us start with \( \psi_1 \land \psi_2 \in \Gamma \). We can start from \(\psi_1\)  and use saturation to produce a world \( (S_1 , \Gamma_1, \Delta_1) \models \psi_1 \), lower that the first world. Similarly with \( \psi_2 \). But then, by the rule of Volpin forcing, \( (S, \Gamma, \Delta) \models \psi_1 \land \psi_2 \).

2. \( \phi = \psi_1 \lor \psi_2 \): 

Essentially the same  argument (both ways)  works as in the case above. 

3. $\phi = \phi_1 \rightarrow \phi_2$: 

If $\phi = \phi_1 \rightarrow \phi_2 \in \Gamma$, then consider  any world $w_1 =(S_1`, \Gamma_1, \Delta_1)$ extending $w= (S`, \Gamma, \Delta)$, and assume that for some world $w_2\leq w_1$, $ w_2 = (S_2, \Gamma_2, \Delta_2 ) \models \phi_1 $. Then by inductive hypothesis,    $\phi_1 \in \Gamma_2$. Then, since also $\phi_1 \rightarrow \phi_2 \in \Gamma_1$ by monotonicity, $w_2 \models \phi_2$ by deductive closure. But this implies by the Volpin forcing that $w \models \phi_1 \rightarrow \phi_2$.

On the other hand, if $\phi = \phi_1 \rightarrow \phi_2$ is not in $\Gamma$, then the bi-theory $(S, \Gamma \cup \{\phi_1\}, \{\phi_2\})$ is $m+1$- consistent.
By saturating it, we obtain a canonical model world $w_1 =(S_1, \Gamma_1, \Delta_1)$ extending $w$, such that $(w_1 \Vdash \phi_1$ and $w_1 \nVdash \phi_2$. Again via saturation starting from  $\emptyset, \phi_1, \emptyset$, we can obtain a world $w_2 \leq w_1$ such that $w_2 \models \phi_1$. But then,   $S, \Gamma, \Delta) \nVdash \phi_1 \rightarrow \phi_2$.

4.  \( \phi = \exists x \psi(x) \): 

The $\exists$ case follows from existential  closure of $(S, \Gamma, \Delta)$: $w= (S, \Gamma, \Delta) \Vdash \exists x \psi(x)$ 
if and only if $w_1 =(S_1, \Gamma_1, \Delta_1) \Vdash \psi(a)$ for some $a \in S_1$. But then by inductive hypothesis \( \psi(a) \in \Gamma_1\) and thus also in \(Gamma\). Now, invoke m-deductive closure to conclude that \( \exists x \psi(x) \in \Gamma)\).

Other way around: assume \( \exists x \psi(x)\ \in \Gamma)\) . Now use existential closure to get \(t\). To conclude, via $m-1$ saturation of $(S_{Omega} \cup \{a\},  \{\psi(a), \emptyset\})$ carve out a smaller world $w_1$ which witness a concrete instance of the existential quantifier. Now, the forcing rule completes the proof.

5 \( \phi = \forall x \psi(x) \): 

If $\phi = \forall x \psi(x)$ is in $\Gamma$, then it is in $\Gamma_1$ for any $(S_1, \Gamma_1, \Delta_1) \in R((S, \Gamma, \Delta))$. Take an arbitrary $a \in S_1$. By deductive closure, $\psi(a) \in \Gamma_1$ (it requires only one step deduction), and thus by induction hypothesis $(S_1, \Gamma_1, \Delta_1) \Vdash \psi(a)$. Therefore, by definition of Volpin forcing, $(S, \Gamma, \Delta) \Vdash \forall x \psi(x)$.

Now let $\forall x \psi(x)$ be in $\Delta$. Let $a$ be a new constant from $S_0 - S$ (non-empty, since our bi-theory is small).  The bi-theory $(S \cup \{a\}, \Gamma, \{\psi(a)\})$ is $m +1$-consistent. Saturate it. We obtain a world $(S_1, \Gamma_1, \Delta_1) \in R((S, \Gamma, \Delta))$ that falsifies $\psi(a)$. Therefore, $\forall x \psi(x)$ is false in $(S, \Gamma, \Delta)$.

\end{proof}

\begin{theorem}[\textbf{Completeness}]
If \( \phi \) is universally true in all Esenin-Volpin models up to depth \( k \), then \( \phi \) is derivable within \( k \) steps.
\end{theorem}

\begin{proof}
We will prove this by contra-position. Assume \( \phi \) is not derivable  within \( k \) steps. 

Given this, consider the bi-theory \( (\emptyset, \emptyset, \{\phi\}) \). Since \( \phi \) is not derivable, this bi-theory is \( k \)-consistent. 

Now, we saturate it using the Saturation Lemma tailored for \( k \)-consistency. This provides us with a world \( w = (S_0, \Gamma_0, \Delta_0) \) in \( M_0 \) such that \( \phi \) is in \( \Delta_0 \).

Given the construction of our canonical model and the Main Semantic Lemma (tailored for \( k \)-completeness), \( w \) does not force \( \phi \) at depth \( k \), i.e., \( w \nvdash_k \phi \).

Since \( w \) was constructed from a \( k \)-consistent bi-theory where \( \phi \) was not derivable, and in \( w \) \( \phi \) is not forced at depth \( k \), it follows that \( \phi \) is not universally true in all Esenin-Volpin models up to depth \( k \).

\end{proof}

\subsection{From axioms to theorems: going deeper}

Given a Esenin-Volpin model \( V \) where \( V \) verifies a finite list of axioms \( T \) at depth \( k_1 \), we aim to investigate the depth at which a formula \( \psi \) is forced, given a \( k_2 \)-step proof that \( T \implies \psi \).

1. By the properties of the Esenin-Volpin model, for each axiom \( \tau \in T \), \( \tau \) is forced at depth \( k_1 \) in \( V \).

2. Given the formal notion of forcing, for any proof of length \( n \) that \( \varphi \rightarrow \psi \), if \( \varphi \) is forced at depth \( k \), then \( \psi \) is forced at depth \( k+n \). Using this rule:

\begin{itemize}
    \item The base case is each \( \tau \in T \) which acts as our \( \varphi \) and is forced at depth \( k_1 \).
    \item As we progress through the \( k_2 \)-step proof, we increase our depth iteratively based on our forcing rule.
\end{itemize}

3. By combining the depths from our intermediate steps, the conclusion \( \psi \) is forced at depth \( k_1 + k_2 \) in \( V \).

Thus, \( \psi \) is forced in \( V \) at depth \( \leq k_1 + k_2 \) (possibly earlier).

\subsection{Taking stock: Deconstruction of  the Term Model}
\begin{quote}
Finitism is the
last refuge of the Platonist. \\
Ed Nelson, see \cite{Nelson}
\end{quote}
 
A key concept in our discussion is the notion of a \textbf{canonical model} for a k-consistent theory. This arises naturally from the proof of completeness: given a k-consistent theory, one can construct a canonical model for it. 

What is remarkable is its construction: \emph{it is like taking the usual term model, a classical object of study in model theory, and deconstructing it into finite fragments of bounded complexity}. These fragments correspond to the different 'levels' of the theory, each one being a consistent set of sentences involving only a limited number of terms. The canonical Esenin-Volpin model is then constructed by 'stacking' only those  finite fragments which are not too deep on top of each other, leading to a k-complete model for the original theory.

\begin{tcolorbox}[colback=black!5!white,colframe=black!75!black,title=Deconstructing N]
   In order to acquire a taste  for this deconstruction, we suggest focusing on the most famous term model in the worlds, $N$. Let us say that you operate within Robinson's Arithmetics $Q$, so you have only the basic operations available. If you try to assemble $N$ from $0$, you generate a list of closed terms, such as $\{ 0, S0, \dots, SS0 + SSSSSSSSSS0, \ldots \}$. As you keep on going, each stage is a node in the Esenin-Volpin frame, where the terms are not necessarily linearly ordered, in fact not even fully ordered. Only available computations decide what is identical, who comes first and who comes next, etc. Furthermore, the very terms which denote arithmetic entities are not given all at once, but are an act of creation (thus, assuming all primitive recursive functions as given as it is the case in Predicative Recursive Arithmetics (PRA) may be satisfying for the standard finitist, but certainly not here. In the end, what one is left with are \emph{finite fragments of the term model}, which may look quite different from the simple picture we have of $N$, or even its finite cuts. 
   
\end{tcolorbox}

\section{Esenin-Volpin Models in Arithmetics and Set Theory }

Armed with Esenin-Volpin structures, let us begin our modest steps into  ultrafinite mathematics. In the following we will restrict ourselves to basic arithmetic and basic set theory. A broader  exploration of the ultrafinite land is left to a future work.

As an appetizer, let us make a simple observation: suppose you work with PA, and you, like most of the math community, believe PA is consistent. In the new jargon, you can rephrase your legitimate belief  as saying: \emph{PA is k-consistent for an arbitrary large k}. Thus, nothing prevents you from picking your favorite large number k, and then create a Esenin-Volpin model of the truncated theory of HA, up to k. We shall not pursue this line of thoughts further for the time being, because it is worth a separate investigations: truncating HA o IZF, or other familiar theories, may reserve some surprises!

Before starting, it is worth highlighting two main strategies, both usable as magic wands, to turn familiar math theories  into manifestly inconsistent yet k-consistent ones:

\begin{tcolorbox}[colback=black!5!white,colframe=black!75!black,title=Miniaturizing Arithmetics or Set Theory]

    \begin{enumerate}
        \item Keeping induction, but adding a new unary predicate which satisfies some inductive steps but to which induction does not apply (This is Parikh's approach in the celebrated paper \cite{parikh}).
        \item Keeping the ground theory induction free, for instance, starting from  Q, and adding some large number (or set) and postulating its inaccessibility.
    \end{enumerate}
\end{tcolorbox}

The key difference between these strategies lies in the Induction Principle: in Ultrafinite Land, Induction is not necessarily banished, rather,  one has to understand it from a different angle: Induction  essentially means -you can have a free lunch for as long as you wish-.  In a world where everything burns resources, it is nice to have it, of course. But it comes at a price, a subtle one: \textit{IP tends to flatten structures}. For instance, if we work in arithmetic and we do not use any induction, comparison between closed terms requires computations, The accepted fact that every closed term can be reduced to its canonical form $SSSSS \ldots 0$ is not necessarily true. So, if we opt for (some) induction, we can overlay our structures with a feasibility predicate which does not obey induction. Or, we may decide to  strip away induction altogether. 

In the next sub-section, we will adopt the second strategy for arithmetics and the first one for a toy set theory. It goes without saying that there could be other ways beyond these two for "doing the trick".

\subsection{Graham's Arithmetic and Esenin-Volpin Models}

We consider a variant of Robinson's Arithmetic $Q$ we shall refer to as  GRAHAM, which is simply $Q$  extended to include all the recursive definitions necessary to talk about $g_{64}$, where $g$ denotes Graham's infamous number. Here is Pakhomov succinct definition:

\begin{definition}[\textbf{Graham Arithmetics}]
Graham Arithmetics, denoted as \(\mathcal{G}\), extends the base system \(Q\) by introducing an additional function \(x \uparrow^y z\) along with the following axioms:
\begin{align*}
    x \uparrow^y 0 &= x, \\
    x \uparrow^0 S(y) &= x \cdot (x \uparrow^0 y), \\
    x \uparrow^{S(z)} S(y) &= x \uparrow^z (x \uparrow^{S(z)} y).
\end{align*}
It should be noted that in this context, \(x \uparrow^y z\) represents the notation \(x \uparrow \dots \uparrow y+1\) with \(z+1\) arrows, instead of \(x \uparrow \dots \uparrow y\) with \(z\) arrows. The change in notation stems from the fact that in \(Q\), we initiate natural numbers with \(0\) as opposed to \(1\).

The numbers \(g_n\) are  represented by the terms:
\begin{align*}
    g_1 &= 3 \uparrow^{3} 2, \\
    g_{n+1} &= 3 \uparrow^{g_n} 2.
\end{align*}
\end{definition}

Now, the k-consistent arithmetic theory can be formally expressed as GRAHAM along with an additional statement, stating that Graham number is inaccessible,  $\forall n, n < g_{64}$. Let us call it $GRAHAM(g_{64})$. 

While this theory is  classically inconsistent, it appears to lack short proofs of inconsistency without a direct appeal to induction. In the Mathoverflow post (see \cite{Mathoverflow1} ) where the author introduced $GRAHAM(g_{64})$ and asked for an estimate of the length of proofs of its inconsistency, two experts, Emil Jerabek and Fedor Pakhomov, provided excellent answers. 

Let's visually represent a (tiny portion of) the canonical Esenin-Volpin model for GRAHAM. Consider the following diagram:

\begin{figure}[!htb]
\centering

\begin{tikzpicture}
    \node[draw,circle] (w1) at (0,0) {$\{0\}$};
    \node[draw,circle] (w2) at (2,2) {$\{0, 1\}$};
    \node[draw,circle] (w3) at (-2,2) {$\{0, 2\}$};
    \node[draw,circle] (w4) at (4,4) {$\{0, 1, 2\}$};
    \node[draw,circle] (w5) at (0,4) {$\{0, 1, 2^2\}$};
    \node[draw,circle] (w6) at (-4,4) {$\{0, 2, 2^{2^2}\}$};
    \node[draw,circle] (w7) at (6,6) {$\{0, 1, 2, 2^{2^2}\}$};
    \node[draw,circle] (w8) at (0,8) {$\{0, 1, 2^{2^{2^2}}, g_{64}\}$};

    \draw[->] (w1) -- (w2);
    \draw[->] (w1) -- (w3);
    \draw[->] (w2) -- (w4);
    \draw[->] (w2) -- (w5);
    \draw[->] (w3) -- (w6);
    \draw[->] (w4) -- (w7);
    \draw[->] (w5) -- (w8);
\end{tikzpicture}

\caption{A tiny  portion of the canonical model for $GRAHAM(g_{64})$. As Tetartion is available, closed terms based on this fast growing function crop up here and there. The theory essentially asserts that $g_{64}$ acts as an infinite number. }
\label{fig:simple_diagram}
\end{figure}
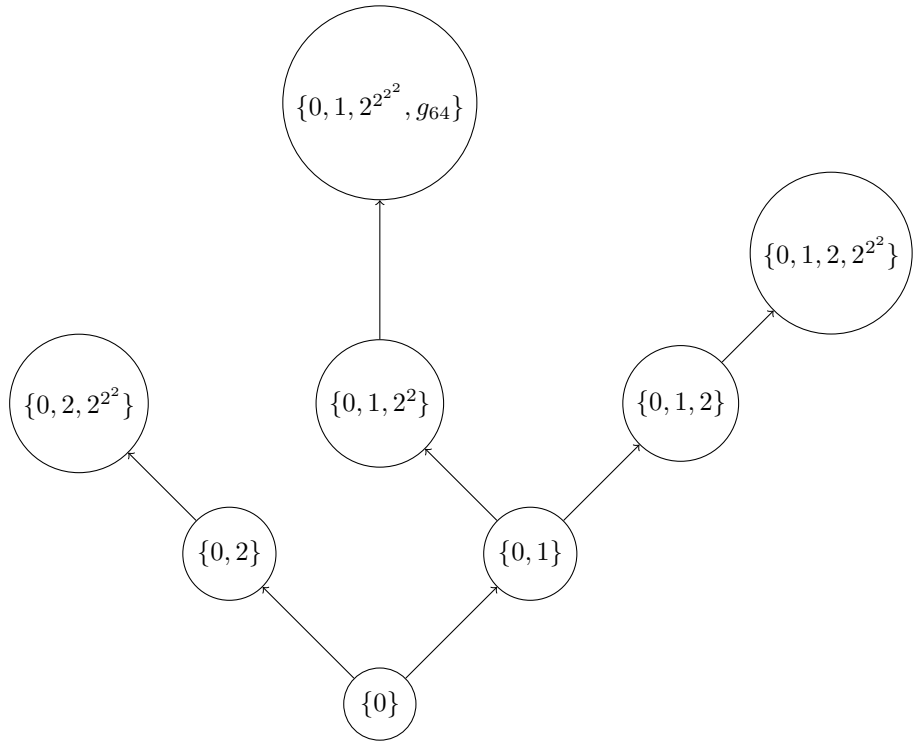

So, how big is $k$ for $GRAHAM(g_{64})$ ?
Pakhomov's estimate for  k is this: \\

\[
 \boxed{(\log_2^*(g_{64}))^{1/N} \leq  k \leq (\ln^*(g_{64}))^{N}}
 \]

where $\ln^*(x)=\min\{n\mid \log_2^n(x)<0\}$
 and N is some reasonable small integer (which could be figured out by a careful examination of the proof).

Incidentally, the proof is interesting on two counts:

\begin{itemize}
\item the Graham term can be replaced as wished, and the proof still stand. So, for instance, one couple replace it with $g_{g_{64}})$ or even more outrageously large integers expressible via the upper arrow or even faster growing primitive recursive functions. 

\item The proof itself uses a model theoretic approach based on  \textbf{fulfillment sequences}, see \cite{Quinsey} , by showing that the theory can be fulfilled, and then proceeding to shorten the fulfilling sequence till the desired minimum length. This is very interesting, because it looks like as if fulfilling sequences are special Esenin-Volpin structures (the exact relation between them is still to be determined). The completeness theorem of the previous section together with the shortening technique, could be used to gauge the consistency radius of an inconsistent theory via model theoretical reasoning. 
\end{itemize}
Even in its simplicity, Graham arithmetics and its Esenin-Volpin models can teach us something: observe that thanks to the powerful tetartion notation, even term representing almost unimaginable numbers appear quite soon in the model.

Of course manipulating, comparing those terms is a completely different matter: at some point in the model two such terms may be incomparable, simply due to the fact that the computation required is too large. 
\subsection{Esenin-Volpin Models of Feasible Set Theory}
Disclaimer: this sub-section is an appetizer at best. A full discussion of set theory within the context of the ultra-finite, due to the richness of the topic, is no minor task. So, we shall opt for a little toy theory, which hopefully shall provide a foretaste of things to come (incidentally, there are other candidates for feasible set theory most notably \cite{lavine}). 

 \textbf{Feasible Set Theory} (FIZF) is a modification of Intuitionistic Zermelo Fraenkel set theory which introduces a new monadic predicate \( F(x) \) to denote the feasibility of the set \( x \). This adaptation is aimed at capturing the notion of sets that can be concretely realized or understood within certain computational boundaries.

FIZF incorporates axioms which ensure that some constructions  performed using feasible sets will also result in other feasible sets.
The absolute minimum is to chose empty set, union, pairing, as they are  the tamest ones. Notice though that the two combined can still show, one step at the time, that \textit{standard finite ordinals are feasible}.  
\begin{enumerate}
    \item \textbf{Empty Set:}  \(F( \emptyset ) \)
    \item \textbf{Union:} If \( F(a) \) and \( F(b) \) , then \( F(a \cup b) \) .
    \item \textbf{Pairing:} If \( F(a) \) and \( F(b) \), then \( F(\{a, b\}) \).

\end{enumerate}

Admittedly, FIZF,  as defined thus far, does not go too far in the vast set theoretic realm. But if we add stability of power set with respect to feasibility, we already cover some ground (note in passing that this addition is like stipulating that the exponential is a feasible operation in arithmetics)
\begin{enumerate}
\item \textbf{Power set:} If \( F(a) \), then \( F(\mathcal{P}(a)) \) 
\end{enumerate}

Furthermore,  we stipulate that certain extraordinarily large yet finite sets are not feasible (we shall recycle Graham's number from the previous sub-section):

\begin{itemize}
    \item \( \neg F(g_{64}) \), ie Graham's number  is not feasible. This underpins the ultrafinitist perspective where even certain finite sets might be beyond concrete realization.
\end{itemize}

This modification of IZF results in an inconsistent theory under classical and intuitionistic  logic, yet FIZF remains $k$-consistent for a very large $k$, essentially via Parikh's arguments in \cite{parikh} (key point: F(x) is a new monadic predicate, not expressible in standard set theory. Ordinal induction does not apply to F(x) ).

As such, we can construct a canonical Esenin-Volpin model for FIZF (actually for a finite fragment of it), whose leaf nodes are somewhat reminiscent of the truncated cone in the von Neumann hierarchy that extends up to the ordinal corresponding to $g_{64}$.  Several sets in the truncated cone will be unfeasible, unless we decide to power up our set theory with the closure of the monadic predicate with respect to other axioms of IZF, such as replacement. The little model will also contain a plethora of sets living above \( g_{64} \), somewhat playing the role of  infinite sets. We shall return to our feasible set theory in \cite{Mannucci4}, where the focus will be on Cantor's infamous definition of infinite set.

\begin{figure}[!htb]
\centering

\begin{tikzpicture}

% Lines for the cone
\draw (0,0) -- (5,7) -- (-5,7) -- cycle;

% Levels
\foreach \y in {0.5,1.5,2.5,3.5}
    \draw[gray,very thin] (-5*\y/7, \y) -- (5*\y/7, \y);
\foreach \y in {5.5,6.5}
    \draw[gray,very thin, dashed] (-5*\y/7, \y) -- (5*\y/7, \y);

% Ordinals
\node[left] at (-5*0.5/7,0.5) {$\varnothing$};
\node[left] at (-5*1.5/7,1.5) {$2$};
\node[left] at (-5*2.5/7,2.5) {$3$};
\node[left] at (-5*3.5/7,3.5) {$\ldots$};
\node[left] at (-5,5) {$g_{64}$};
\node[left] at (-5,6) {$\vdots$};

% Spine of set theory
\draw[black, very thick] (0,0) -- (0,7);

% Transfinite label
\node[align=center] at (0,6) {transfinite\\ below \( \omega_0 \)};

% Introducing gaps or holes in the cone
\fill[black] (-1,2) circle (0.1);
\fill[black] (1,3) circle (0.1);
\fill[black] (-2.5,4) circle (0.1);

\end{tikzpicture}

\caption{The truncated universe up to Grapham's number. Every set above the graham level is effectively transfinite. Notice that the feasible part has "holes", meaning sets which are not forced feasible inside the model}
\label{fig:simple_diagram}
\end{figure}
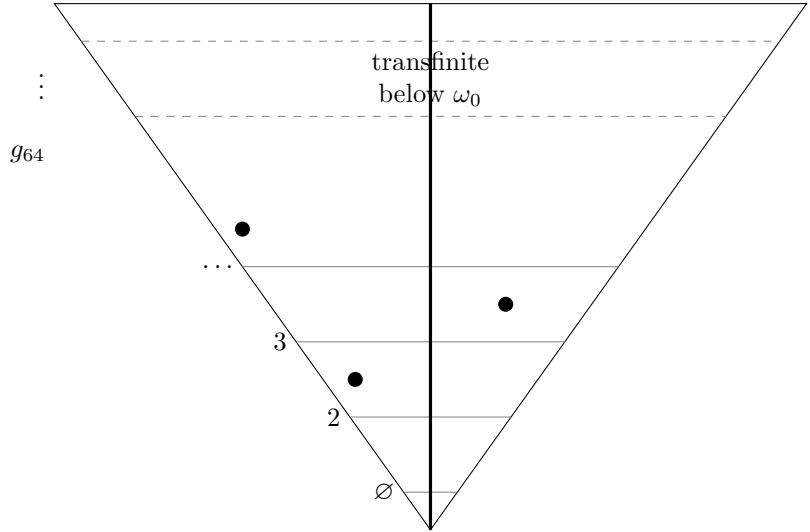

\section{ Fuzziness is in the Depth}
The Reader of \cite{MannucciCherubin} may be a bit confused: where are the fuzzy initial segments of arithmetics, where is  the semantics based on a fuzzification of Tarskian model theory? 
At first sight there is a certain disconnect. Actually, the "fuzziness" is,  as it were,  embedded in the very shape of Volpin forcing. How? The key is to think of the depth at which a certain sentence is forced.

Given a Esenin-Volpin model of maximal  height \( d \) and a leaf node, we associate to each formula  \(\phi\) forced somewhere along the path leading to the leaf a value, which corresponds to the "inverted" minimal depth at which it is forced. Specifically, if \(\phi\) is forced at depth \( d_{\phi} \), the associated value is given by \( value(\phi) = \frac{d - d_{\phi}}{d} \), so that the values range from \( 0 \) to \( 1 \).
The intuitive interpretation of such a value is: the higher the formula shows up, the more costly it is, and therefore the less valid (as a reference, suppose someone would come along and tell you that he/she has a proof of the Riemann Hypothesis in $10^{10^{10}}$ pages. Your immediate reaction would likely be: this fellow is a phony, or simply you will wait till someone else will exhibit a more manageable proof).

\subsubsection*{Composite Formulas}

\begin{itemize}
    \item \textbf{Conjunction (\(\land\))}: If \(\phi\) is forced at depth \( d_{\phi} \) and \(\psi\) at \( d_{\psi} \), then \( \phi \land \psi \) is forced at depth \( \max(d_{\phi}, d_{\psi}) + 1 \). \\
    Associated value: \( value(\phi \land \psi) = \frac{d - (\max(d_{\phi}, d_{\psi}) + 1)}{d} \).

    \item \textbf{Disjunction (\(\lor\))}: If either \(\phi\) is forced at depth \( d_{\phi} \) or \(\psi\) at \( d_{\psi} \), then \( \phi \lor \psi \) is forced at depth \( \min(d_{\phi}, d_{\psi}) + 1 \). \\
    Associated value: \( value(\phi \lor \psi) = \frac{d - (\min(d_{\phi}, d_{\psi}) + 1)}{d} \).

 \item \textbf{Implication (\(\rightarrow\))}: Let the value of \( \psi \) be \( d_{\psi} \) and the value of \( \phi \) be \( d_{\phi} \). Then, the value of \( \phi \rightarrow \psi \), denoted as \( value(\phi \rightarrow \psi) \), is determined by the comparison of \( d_{\phi} \) and \( d_{\psi} \). If \( d_{\phi} > d_{\psi} \), then \( value(\phi \rightarrow \psi) = d_{\phi} \). If \( d_{\phi} \leq d_{\psi} \), then \( value(\phi \rightarrow \psi) = d_{\psi} \). \\
    Associated value: \( value(\phi \rightarrow \psi) = \begin{cases} d_{\phi} & \text{if } d_{\phi} > d_{\psi}, \\ d_{\psi} & \text{if } d_{\phi} \leq d_{\psi}. \end{cases} \)
.

    \item \textbf{Universal Quantifier (\(\forall\))}: 
    If for every \( a \in D_w \), \( \phi(a) \) is forced at some depth \( d\), then the universal quantification \( \forall x \, \phi(x) \) is forced one step up\\
Associated value: \( value(\forall x \, \phi(x)) = \frac{d - (1 + \max \{d_{\phi(a)} | a \in D_w\})}{d} \).
\end{itemize}

Summing up, to derive a fuzzy model from a Esenin-Volpin one, we follow the  series of steps listed below:

\begin{enumerate}
    \item \textbf{Selection of a Leaf Node:} Choose a leaf node from the Esenin-Volpin model. The leaf represents a consistent snapshot of the logical world.
    
    \item \textbf{Defining the Domain:} The domain of our fuzzy model is the domain of the chosen leaf node, providing a clear set of terms.
    
    \item \textbf{Assigning Fuzzy Values:} Utilize the depth of each term and sentences that are forced at the leaf to provide them a value in the fuzzy model. The depth, in essence, offers a "degree of truth" for each proposition.
\end{enumerate}

The result is a fuzzy model that captures the graded nuances of the logical world as represented in the Esenin-Volpin model (see Figure 3). The Reader may still be a bit unhappy: whereas in the Manifesto there were many nuances of fuzziness, here the growth is linear. Why so? To begin with, our notion of complexity of proofs is the simplest one, and thus are its associated Esenin-Volpin models. The purpose of this article is simply to showcase what can be done, but it is far from  the most general setup one can conceive. Once a notion of \textbf{Truth Transfer Policy} (see again the Manifesto) is in place, it is also possible to generalize Esenin-Volpin models to provide for a more flexible notion of truth propagation, and thus of fuzziness. 

\usetikzlibrary{positioning}

\begin{figure}[!htb]
\centering

\begin{tikzpicture}

% Esenin-Volpin Tree
\node[circle, draw, label=above:Root] (root) at (0,0) {};
\node[circle, draw] (a) at (-1.5,1) {};
\node[circle, draw] (b) at (1.5,1) {};
\node[circle, draw] (c) at (-2.5,2) {};
\node[circle, draw] (d) at (-0.5,2) {};
\node[circle, draw] (e) at (0.5,2) {};
\node[circle, draw] (f) at (2.5,2) {};
\node[circle, draw, fill=black] (leaf) at (-3,3) {}; % Highlighted leaf

\draw (root) -- (a) -- (c) -- (leaf);
\draw (root) -- (b) -- (f);
\draw (a) -- (d);
\draw (b) -- (e);

% Fuzzy Set Expansion from Leaf
\begin{scope}[yshift=8cm, xshift=-3cm]
    \draw[fill=gray!20, opacity=0.7] (-3,-1) circle [radius=2];
    \draw[fill=gray!30, opacity=0.7] (-3,-1) circle [radius=1.5];
    \draw[fill=gray!40, opacity=0.7] (-3,-1) circle [radius=1];
    \draw[fill=gray!70, opacity=0.7] (-3,-1) circle [radius=0.5];

    \node at (-3, 2) {Fuzzy Model Expansion};
\end{scope}

% Dashed Line from Leaf to Fuzzy Expansion
\draw[dashed, black, thick] (leaf) -- (-5,5);

\end{tikzpicture}

\caption{ Expansion of a leaf node. The concentric circles represent degrees of fuzziness which grows with depth at which elements and properties arise}
\label{fig:simple_diagram}
\end{figure}
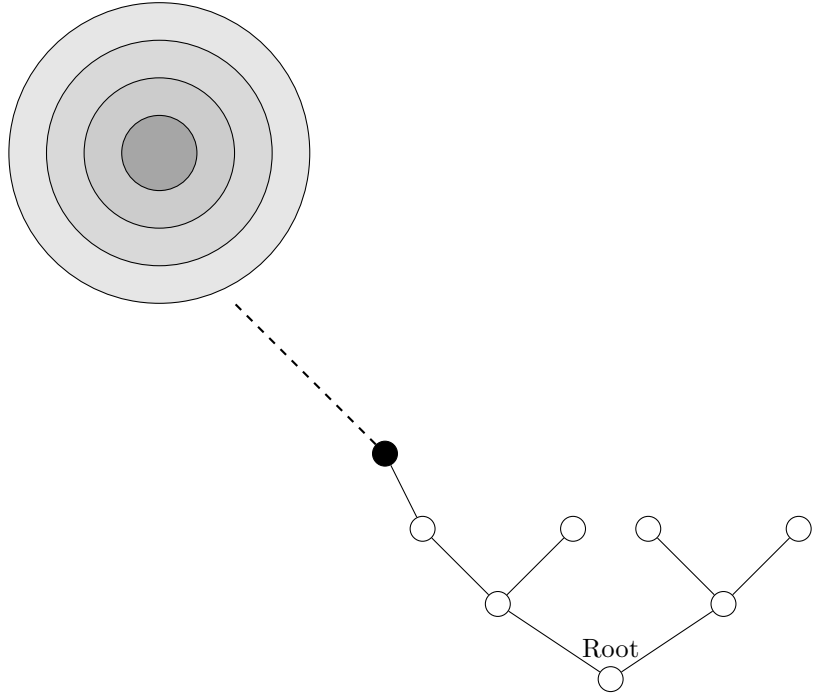

\section{Future Work}

Our investigation of Esenin-Volpin models in this paper sets - we hope- the groundwork for several exciting avenues for future exploration. We outline just a few topics  which seem to us of relevance.

\subsection{Chains of Inaccessibles  in Ultra-finite Arithmetic and Set Theory}
The ultimate dream behind this Program is to miniaturize Cantor's extraordinary theory of the transfinite (what we have referred  to as \textbf{Cantorian Nanotechnology}), all the while  remaining below $\aleph_0$. To this effect, \emph{one must be able to reproduce different orders of infinity}, of course in  the restricted sense described here. 
In creating a theory reminiscent of the various orders of inaccessibility but within the confines of the finite, we could consider extending Robinson's arithmetic \( \mathcal{Q} \) by introducing axioms that hint at the "inaccessibility" of certain large numbers through specific growing functions.

Consider a hierarchy of Knuth's up-arrow notations (any fast growing hierarchy would do, but as we have already used this one implicitly in Graham's arithmetic, let us stick to it for the time being. We define a sequence of numbers, named \(k_i\), as follows:

\begin{enumerate}
    \item \(k_0 = 3 \uparrow\uparrow 10\) — This value is set such that \(k_0\) is unreachable by \(3 \uparrow n\), since it is defined in terms of the next notation.
    \item For all \(i > 0\), \(k_i = 3 \uparrow^{i+2} k_{i-1}\). This recursively ensures that \(k_i\) is out of the reach of both \(3 \uparrow^{i+1} n\) and all its preceding notations.
\end{enumerate}

Now, our finite theory, \( \mathcal{T} \), is articulated through the following axioms:

\[
\mathcal{T} = \mathcal{Q} \cup \bigcup_{i=0}^{N-2} \left\{ \forall n \, (3 \uparrow^{i+1} n < k_i) \right\}
\]

Here, each axiom \( \forall n \, (3 \uparrow^{i+1} n < k_i) \) posits that the \(i^{th}\) up-arrow notation is incapable of reaching its respective "inaccessible" number \(k_i\). This provides us with a finite list of named inaccessible numbers, where each one is defined using the next up-arrow notation applied to the previous inaccessible, thereby ensuring a hierarchy of rapid growth.

The true consistency strength  of this theory (and thus the depth of the Volpin model needed to validate it) remains open for rigorous analysis.

For the time being, we just want to emphasise that this process of miniaturization  is rather simple-minded. To be able to squeeze the full   (or even a substantial part of ) theory of Large Cardinals, requires a substantial work. For instance, some of the cardinals are defined by sophisticated infinite combinatorics properties, so a vocabulary should be built to translate those properties  down to our minuscule world. Other cardinals are definite in terms of meta-mathematical  properties, and to replicate them we need to develop a meta-mathematics inside Volpin structures, a no minor feat.

\subsection{Esenin-Volpin Models as  boards for Bounded Resource Logic Games}

Logic game semantics provides a dynamic approach to understanding the intricate relationships between structures and formulas in model theory. Central to this approach are two-player games, where the players, typically referred to as the Verifier and the Falsifier, aim to establish the truth or falsehood of a formula in a given structure. Esenin-Volpin models, inspired by an ultrafinitistic viewpoint, present a unique twist on traditional model theory. Rather than utilizing standard infinite game semantics, one can delve into this ultrafinitistic landscape using \textbf{bounded resource logical games}. Such games restrict the usual infinite or unbounded strategies, reflecting the underlying philosophy of Esenin-Volpin models centering on finite  albeit possibly very large, constructs. By tailoring logic games to incorporate bounded resources, we can pave the way to an enriched understanding of the ultrafinite.

\subsection{The  Paradoxes of Strict Finitism}

\begin{quote}
"Weir argues (2010; 2016) that finitist formalism is not only extremely radical, it is incoherent. The reason is the all-pervasive role of abbreviation which generates complex tokens where most of their sub-parts will never exist, for example abbreviations which create numerals naming (speaking with the platonist) arbitrarily high numbers. As a result, strict finitist specifications cannot be given in the usual inductive fashion, as the intersection of all inductive sets containing the base set and closed under the complexity-forming operations. And this in turn means we cannot prove even very simple facts about wffs and proof." see \cite{weir}
\end{quote}, 

Against the critique above, Esenin-Volpin models offer an interesting approach. Instead of providing a rigid limit to feasibility, they incorporate a dynamic horizon which is capable of expanding to accommodate mathematical and meta-mathematical constructions of arbitrary complexity.

In this light, the apparent paradoxes of strict finitism, such as the problems with abbreviation and the existence of simple sentences with large finite proofs, can be understood and managed within the flexible and hierarchical framework provided by Esenin-Volpin models. 

To fully implement this solution, it is clear what the first steps should be: \textbf{do meta-mathematics inside Volpin models}.

\subsection{Algebraic and Modal logic of  feasible provability}
It is well known that intuitionistic logic can be reinterpreted as an extension of classical logic via the modality $S_4$: $\Box \psi$ will then mean $\psi$ is true constructively, ie there is a  valid procedure to validate the sentence. We have seen that, for an assigned measure of complexity $k$, $T \vdash _{k} \phi$ is equivalent to $T \vDash _{k} \phi$ with respect to the Esenin-Volpin interpretation. It would be expected that one can have a collection of modal operators whose intended meaning is- I can feasibly prove $\phi$ with a k-proof-. There is a problem though: whereas, in algebraic terms, the modality for intuitionism is regular (ie preserves finite meets) , that is not anymore true here. We are stepping into a new territory, as far as algebraic logic goes. This same issue becomes pertinent in the next subsection as well.

\subsection{ From Esenin-Volping Forcing to The Ultra-finite Universe of Discourse}

The development of a categorical version of Esenin-Volpin semantics presents an intriguing possibility. While Kripke semantics for intuitionistic logic can be subsumed within  topos theory semantics, a direct adoption  is not  suitable due to topos' internal (higher order) intuitionistic logic. Thus, a new categorical framework that accommodates the distinct nature of Esenin-Volpin features, notably the new forcing rules,   needs to be conceptualized and implemented.

One possible approach is to generalize Esenin-Volpin models by creating presheaves from a small finite category into the category of finite sets of bounded maximal size. However, this generalization introduces the challenge of redefining operations to align with the new forcing conditions, as the standard manipulations on sub-objects would not suffice (just think of the basic connectives: to mimic the Volpin forcing rule for the meet, standard sub-objects pullbacks won't do). 

The next installment in the series (\cite{Mannucci}) will strive to create a ultrafinite universe of discourse, albeit from a somewhat different angle: leveraging the Effective Topos and categorical realizability, rather than coming from Kripke-Cohen forcing.

\section*{Acknowledgments}

I would  like to dedicate this work to the memory of the late Professor David Isles (1935-2023), a logician with a strong ultrafinitistic bend, and a contributor with several useful insights (see for instance \cite{Isles1} and \cite{Isles2}. He spent a long time striving to turn the grandiose intuitions of Esenin-Volpin into actual logic ). Isles, a very kind and humble man, also shared with me via email key insights and questions, which found their way into my work. Of course, all responsibilities and possible misunderstandings  are entirely mine.

\end{document}